\documentclass[10pt]{amsart}
\usepackage[
paper=a4paper,
text={135mm,216mm},centering
]{geometry}
\usepackage{amssymb,amsxtra}
\usepackage{amsmath,float}
\usepackage{pb-diagram,pb-xy}
\usepackage[all]{xy}
\usepackage{xcolor}
\definecolor{todo-background-color}{gray}{0.95}
\usepackage[
textwidth=1.1in,
backgroundcolor=todo-background-color,
bordercolor=black,
linecolor=black,
textsize=footnotesize
]{todonotes}
\usepackage{hyperref}

\iftrue

\makeatletter
\def\@settitle{%
	\vspace*{-10pt}
	\begin{flushleft}%
		\LARGE\bfseries
		\strut\@title\strut
	\end{flushleft}%
}
\def\@setauthors{%
	\begingroup
	\def\thanks{\protect\thanks@warning}%
	\trivlist
	\raggedright
	\large \@topsep27\p@\relax
	\advance\@topsep by -\baselineskip
	\item\relax
	\author@andify\authors
	\def\\{\protect\linebreak}%
	\authors
	\ifx\@empty\contribs
	\else
	,\penalty-3 \space \@setcontribs
	\@closetoccontribs
	\fi
	\normalfont
	\endtrivlist
	\endgroup
}
\def\@setaddresses{\par
	\nobreak \begingroup
	\small\raggedright
	\def\author##1{\nobreak\addvspace\smallskipamount}%
	\def\\{\unskip, \ignorespaces}%
	\interlinepenalty\@M
	\def\address##1##2{\begingroup
		\par\addvspace\bigskipamount\noindent
		\@ifnotempty{##1}{(\ignorespaces##1\unskip) }%
		{\ignorespaces##2}\par\endgroup}%
	\def\curraddr##1##2{\begingroup
		\@ifnotempty{##2}{\nobreak\noindent\curraddrname
			\@ifnotempty{##1}{, \ignorespaces##1\unskip}\/:\space
			##2\par}\endgroup}%
	\def\email##1##2{\begingroup
		\@ifnotempty{##2}{\nobreak\noindent E-mail address%
			\@ifnotempty{##1}{, \ignorespaces##1\unskip}\/:\space
			\ttfamily##2\par}\endgroup}%
	\def\urladdr##1##2{\begingroup
		\def~{\char`\~}%
		\@ifnotempty{##2}{\nobreak\noindent\urladdrname
			\@ifnotempty{##1}{, \ignorespaces##1\unskip}\/:\space
			\ttfamily##2\par}\endgroup}%
	\addresses
	\endgroup
	\global\let\addresses=\@empty
}
\def\@setabstracta{%
	\ifvoid\abstractbox
	\else
	\skip@17pt \advance\skip@-\lastskip
	\advance\skip@-\baselineskip \vskip\skip@
	\box\abstractbox
	\prevdepth\z@ % because \abstractbox is a vtop
	\vskip-28pt
	\fi
}
\renewenvironment{abstract}{%
	\ifx\maketitle\relax
	\ClassWarning{\@classname}{Abstract should precede
		\protect\maketitle\space in AMS document classes; reported}%
	\fi
	\global\setbox\abstractbox=\vtop \bgroup
	\normalfont\small
	\list{}{\labelwidth\z@
		\leftmargin0pc \rightmargin\leftmargin
		\listparindent\normalparindent \itemindent\z@
		\parsep\z@ \@plus\p@
		
	}%
	\item[\hskip\labelsep\bfseries\abstractname.]%
}{%
	\endlist\egroup
	\ifx\@setabstract\relax \@setabstracta \fi
}

\def\ps@headings{\ps@empty
	\def\@evenhead{%
		\setTrue{runhead}%
		\normalfont\scriptsize
		\rlap{\thepage}\hfill
		\def\thanks{\protect\thanks@warning}%
		\leftmark{}{}}%
	\def\@oddhead{%
		\setTrue{runhead}%
		\normalfont\scriptsize
		\def\thanks{\protect\thanks@warning}%
		\rightmark{}{}\hfill \llap{\thepage}}%
	\let\@mkboth\markboth
}\ps@headings

\def\section{\@startsection{section}{1}%
	\z@{-1.4\linespacing\@plus-.5\linespacing}{.8\linespacing}%
	{\normalfont\bfseries\Large}}
\def\subsection{\@startsection{subsection}{2}%
	\z@{-.8\linespacing\@plus-.3\linespacing}{.5\linespacing\@plus.2\linespacing}%
	{\normalfont\bfseries\large}}
\def\subsubsection{\@startsection{subsubsection}{3}%
	\z@{.7\linespacing\@plus.2\linespacing}{-1.5ex}%
	{\normalfont\itshape}}
\def\paragraph{\@startsection{paragraph}{4}%
	\z@{.7\linespacing\@plus.2\linespacing}{-1.5ex}%
	{\normalfont\itshape}}
\def\@secnumfont{\bfseries}

\renewcommand\contentsnamefont{\bfseries}
\def\@starttoc#1#2{\begingroup
	\setTrue{#1}%
	\par\removelastskip\vskip\z@skip
	\@startsection{}\@M\z@{\linespacing\@plus\linespacing}%
	{.5\linespacing}{%\centering
		\contentsnamefont}{#2}%
	\ifx\contentsname#2%
	\else \addcontentsline{toc}{section}{#2}\fi
	\makeatletter
	\@input{\jobname.#1}%
	\if@filesw
	\@xp\newwrite\csname tf@#1\endcsname
	\immediate\@xp\openout\csname tf@#1\endcsname \jobname.#1\relax
	\fi
	\global\@nobreakfalse \endgroup
	\addvspace{32\p@\@plus14\p@}%
	\let\tableofcontents\relax
}
\def\contentsname{Contents}
\def\l@section{\@tocline{2}{.5ex}{0mm}{5pc}{}}
\def\l@subsection{\@tocline{2}{0pt}{2em}{5pc}{}}
\makeatother

\fi
\def\Z{\mathbb{Z}}

\def\Q{\mathbb{Q}}
\def\R{\mathbb{R}}

\def\d{\partial}

\def\+{\oplus}

\theoremstyle{plain}
\newtheorem{theorem}{Theorem}[section]
\newtheorem{theoremalpha}{Theorem}
\newtheorem{corollaryalpha}[theoremalpha]{Corollary} 
\newtheorem{proposition}[theorem]{Proposition}
\newtheorem{lemma}[theorem]{Lemma}
\newtheorem{corollary}[theorem]{Corollary}

\newtheorem*{zerosurgeryobstruction}{Zero Surgery Obstruction}
\newtheorem*{Homology Cobordism Invariance}{Homology Cobordism Invariance}

\theoremstyle{definition}

\newtheorem{question}{Question}

\newtheorem{remark}[theorem]{Remark}
\newtheorem*{acknowledgement}{Acknowledgements}
\numberwithin{equation}{section}
\newtheorem*{organization}{Organization of the paper}

\DeclareMathOperator{\arf}{Arf}
\newcommand{\zee}{{\mathbb Z}}

\def\to{\mathchoice{\longrightarrow}{\rightarrow}{\rightarrow}{\rightarrow}}
\makeatletter 
\newcommand{\shortxra}[2][]{\ext@arrow 0359\rightarrowfill@{#1}{#2}}
\def\longrightarrowfill@{\arrowfill@\relbar\relbar\longrightarrow}
\newcommand{\longxra}[2][]{\ext@arrow 0359\longrightarrowfill@{#1}{#2}}
\renewcommand{\xrightarrow}[2][]{\mathchoice{\longxra[#1]{#2}}%
	{\shortxra[#1]{#2}}{\shortxra[#1]{#2}}{\shortxra[#1]{#2}}}
\makeatother

\begin{document}

\title[3-manifolds that cannot be obtained by 0-surgery on a knot]
{Irreducible 3-manifolds that cannot be obtained by 0-surgery on a knot}


%    author one information
\author{Matthew Hedden}
\address{Department of Mathematics\\
	Michigan State University\\
	MI 48824\\
	USA
}
%\curraddr{}
\email{mhedden@math.msu.edu}
%\thanks{}

%    author two information

\author{Min Hoon Kim}
\address{
	School of Mathematics\\
	Korea Institute for Advanced Study \\
	Seoul 02455\\
	Republic of Korea
}
\email{kminhoon@kias.re.kr}

\author{Thomas E.\ Mark}
\address{
	Department of Mathematics\\
	University of Virginia\\
	VA 22903\\
	USA
}
\email{tmark@virginia.edu}

\author{Kyungbae Park}
\address{
	Department of Mathematical Sciences\\
	Seoul National University \\
	Seoul 08826\\
	Republic of Korea
}
\email{kyungbaepark@snu.ac.kr}

%\author{}
%\address{}
%\curraddr{}
%\email{}
%\thanks{}

%    \subjclass is required.

\date{}

\dedicatory{}

%    Abstract is required.
\begin{abstract}
	We give two infinite families of examples of closed, orientable, irreducible 3-manifolds $M$ such that $b_1(M)=1$ and $\pi_1(M)$ has weight 1, but $M$ is not the result of Dehn surgery along a knot in the 3-sphere. This answers a question of Aschenbrenner, Friedl and Wilton, and provides the first examples of irreducible manifolds with $b_1=1$ that are known not to be surgery on a knot in the 3-sphere.  One family consists of Seifert fibered 3-manifolds, while each member of the other family is not even homology cobordant to any Seifert fibered 3-manifold. None of our examples are homology cobordant to any manifold obtained by Dehn surgery along a knot in the 3-sphere.
\end{abstract}

\maketitle


\section{Introduction}
It is a well-known theorem of Lickorish \cite{Lickorish:1962-1} and Wallace \cite{Wallace:1960-1} that every closed, oriented 3-manifold is obtained by Dehn surgery on a link in the three-sphere. This leads one to wonder how the complexity of a $3$-manifold is reflected in the links which yield it through surgery, and conversely.  A natural yet difficult goal in this vein is to determine the minimum number of components of a link on which one can perform surgery to produce a given 3-manifold.  In particular, one can ask which 3-manifolds are obtained by Dehn surgery on a \emph{knot} in $S^3$.  If, following \cite{Auckly:1997-1}, we define the {\it surgery number} $DS(Y)$ of a closed 3-manifold $Y$ to be the smallest number of components of a link in $S^3$ yielding $Y$ by (Dehn) surgery, we ask for conditions under which $DS(Y) >1$.

The fundamental group provides some information on this problem. Indeed, if a closed, oriented 3-manifold $Y$ has $DS(Y) = 1$, then the van Kampen theorem implies that $\pi_1(Y)$ is normally generated by a single element (which is represented by a meridian of $K$). In particular, $\pi_1(Y)$ has weight one and $H_1(Y;\Z)$ is cyclic. (Recall that the \emph{weight} of a group $G$ is the minimum number of normal generators of $G$.)

A  more sophisticated topological obstruction to being surgery on a knot comes from essential 2-spheres in 3-manifolds. While Dehn surgery on a knot can produce a non-prime 3-manifold, the \emph{cabling conjecture} \cite[Conjecture~A]{Acuna-Short:1986} asserts that this is quite rare and occurs only in the case of $pq$-surgery on a $(p,q)$-cable knot.  It would imply, in particular, that a non-prime 3-manifold obtained by surgery on a knot in $S^3$ has only two prime summands, one of which is a lens space.  Deep work of   Gordon-Luecke \cite[Corollary~3.1]{Gordon-Luecke:1989-1} and Gabai \cite[Theorem~8.3]{Gabai:1987-3} verify this in the case of  homology  spheres and homology $S^1\times S^2$'s, respectively, showing more generally that  if such a manifold is obtained by surgery on a non-trivial knot, then  $Y$ is irreducible.

It is natural to ask whether these conditions are sufficient to conclude that $Y$ is obtained from $S^3$ by Dehn surgery on a knot. In the case of homology 3-spheres, Auckly \cite{Auckly:1997-1} used  Taubes'  end-periodic diagonalization theorem \cite[Theorem~1.4]{Taubes} to give examples of hyperbolic, hence irreducible, homology 3-spheres with $DS(Y)>1$. It remains  unknown, however, if any of Auckly's examples have weight-one fundamental group. More recently, Hom, Karakurt and Lidman \cite{Hom-Karakurt-Lidman:2016-1} used Heegaard Floer homology to obstruct infinitely many irreducible Seifert fibered  homology 3-spheres with  weight-one fundamental groups from being obtained by Dehn surgery on a knot.   In \cite{Hom-Lidman:2016-1}, Hom and Lidman gave infinitely many such hyperbolic examples, as well as infinitely many examples with arbitrary JSJ decompositions. Currently, we do not know whether the examples of \cite{Hom-Lidman:2016-1} have weight-one fundamental groups or not.

It is interesting to note, however, that a longstanding open problem of Wiegold (\cite[Problem 5.52]{Mazurov-Khukhro:2014-1} and \cite[Problem 15]{Gersten:1987-1}) asks whether {\it every} finitely presented perfect group has weight one. The question would be answered negatively if there is a homology 3-sphere whose fundamental group has weight $\geq 2$. 

Using $\Q/\Z$-valued linking form and  their surgery formulae for Casson invariant, Boyer and Lines \cite[Theorem~5.6]{Boyer-Lines:1990-1} gave infinitely many irreducible homology lens spaces which have weight-one fundamental group, but are not obtained by Dehn surgery on a knot. In \cite{Hoffman-Walsh:2015-1}, Hoffman and Walsh gave infinitely many hyperbolic examples of this sort.

For the case that $Y$ is a homology $S^1\times S^2$, significantly less is known.  Aschenbrenner, Friedl and Wilton \cite{Aschenbrenner-Friedl-Wilton:2015-1} asked the following question. 

\begin{question}[{\cite[Question 7.1.5]{Aschenbrenner-Friedl-Wilton:2015-1}}]\label{question:AFW}Let $M$ be a closed, orientable, irreducible $3$-manifold such that $b_1(M) = 1$ and $\pi_1(M)$ has weight $1$. Is $M$ the result of Dehn surgery along a knot in $S^3$?
\end{question}

Note that if $M$ as in the question does arise from surgery on a knot in $S^3$ then necessarily the surgery coefficient is zero.

The purpose of this paper is to give two families of examples that show the answer to Question \ref{question:AFW} is negative. The first family shows that there exist homology $S^1\times S^2$'s not smoothly homology cobordant to any Seifert manifold or to zero surgery on a knot; we recall that two closed, oriented 3-manifolds $M$ and $N$ are \emph{homology cobordant} if there is a smooth oriented cobordism $W$ between them for which the inclusion maps $M\hookrightarrow W \hookleftarrow N$ induce isomorphisms on integral homology.

\begin{theoremalpha}\label{theorem:A} The family of closed, oriented  $3$-manifolds $\{M_k\}_{k\geq1}$ described by the surgery diagram in Figure \ref{figure:Hedden-Mark} satisfies the following.
	\begin{enumerate}		
		\item $M_k$ is irreducible with first homology $\zee$ and $\pi_1(M_k)$ of weight $1$.
		\item $M_k$ is not the result of Dehn surgery along a knot in $S^3$.
		\item $M_k$ is not homology cobordant to Dehn surgery along a knot in $S^3$.
		\item\label{item:Seifert} $M_k$ is not homology cobordant to any Seifert fibered $3$-manifold.
		
		\item $M_k$ is not homology cobordant to $M_l$ if $k\neq l$.
	\end{enumerate}
\end{theoremalpha}
\begin{figure}[htb!]
	\centering
	\includegraphics[scale=1]{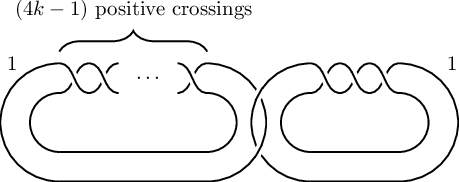}
	\caption{A surgery diagram of $M_k$ ($k\geq 1)$.}
	\label{figure:Hedden-Mark}
\end{figure}

The first property of $M_k$ is relatively elementary; in particular it follows from some general topological results about spliced manifolds. As we show in the next section, any splice of non-trivial knot complements in the 3-sphere is irreducible and has weight one fundamental group, from which our claims about $M_k$ will follow.

To show that $M_k$ is not the result of Dehn surgery along a knot in $S^3$, we use a Heegaard Floer theoretic obstruction developed by 
Ozsv\'{a}th and Szab\'{o}  in \cite{Ozsvath-Szabo:2003-2}. They showed that certain numerical  ``correction terms" $d_{1/2}$ and $d_{-1/2}$ satisfy 
\begin{equation}\label{dconstraints}
d_{1/2}(M)\leq \tfrac{1}{2}\quad\mbox{and}\quad d_{-1/2}(M)\geq -\tfrac{1}{2}
\end{equation}
whenever $M$ is obtained from 0-surgery on a knot in $S^3$ (see Theorem~\ref{thm:OS}). We will show that $d_{-1/2}(M_k)=-\frac{5}{2}$, and hence $M_k$ is not the result of Dehn surgery on a knot in~$S^3$.     The correction terms are actually invariants of homology cobordism, from which it follows that none of the $M_k$ are even homology cobordant to surgery on a knot in~$S^3$. This feature of our examples distinguishes it from the analogous gauge and Floer theoretic results for homology spheres mentioned above.  Indeed, the techniques of Auckly or Hom, Lidman, Karakurt are not invariant under homology cobordism; in the former, this is due to a condition on $\pi_1$ in Taubes' result on end periodic manifolds, and in the latter because the reduced Floer homology is not invariant under homology cobordism (though see \cite{Hendricks-Hom-Lidman:2018} for some results in that direction).

To show that our examples $M_k$ are not homology cobordant to any Seifert fibered 3-manifold, we prove a general result, Theorem~\ref{theorem:dofSeifert}, about the correction terms of Seifert fibered 3-manifold $M$ with first homology $ \Z$: we show that any Seifert manifold with the homology of $S^1\times S^2$ satisfies the same constraints \eqref{dconstraints} as the result of 0-surgery does. Part $(4)$ of our theorem immediately follows.  We remark that it was only recently shown by Stoffregen (preceded by unpublished work of Fr{\o}yshov) that there exist homology 3-spheres that are not homology cobordant to Seifert manifolds, or equivalently that not every element of the integral homology cobordism group is represented by a Seifert manifold. To be precise, Stoffregen showed in \cite[Corollary~1.11]{Stoffregen:2015-1} that $\Sigma(2,3,11)\#\Sigma(2,3,11)$ is not homology cobordant to any Seifert fibered homology 3-sphere by using homology cobordism invariants from Pin(2)-equivariant Seiberg-Witten Floer homology.    

\subsubsection*{Hyperbolic examples}

For any closed, orientable 3-manifold $M$ with a chosen Heegaard splitting, Myers gives an explicit homology cobordism from $M$ to a hyperbolic, orientable 3-manifold  \cite{Myers:1983-1}. By using these homology cobordisms, we can obtain hyperbolic, orientable 3-manifolds $Z_k$ with first homology $\Z$ which are homology cobordant to~$M_k$. Since $d_{-1/2}$ is a homology cobordism invariant, $Z_k$ is also not the result of Dehn surgery along a knot in $S^3$ by Theorem~\ref{thm:OS}. 

\begin{corollaryalpha}There is a family of closed, orientable irreducible $3$-manifolds $\{Z_k\}_{k\geq 1}$ satisfying the following.
	\begin{enumerate}
		\item $Z_k$ is hyperbolic with first homology $\zee$.
		\item $Z_k$ is not the result of Dehn surgery along a knot in $S^3$.
		\item $Z_k$ is not homology cobordant to any Seifert fibered $3$-manifold.
		
		\item $Z_k$ is not homology cobordant to $Z_l$ if $k\neq l$.
	\end{enumerate}
\end{corollaryalpha}
Myers' cobordisms may not preserve the weight of the fundamental groups at hand. If $\pi_1(Z_k)$ has weight one, then $Z_k$ would provide a negative answer to the following question.

\begin{question}Let $M$ be a closed, orientable, hyperbolic $3$-manifold with $b_1(M) = 1$ and $\pi_1(M)$ of weight $1$. Is $M$ the result of Dehn surgery along a knot in $S^3$?
	
\end{question}
We remark that the question is also open for integral homology 3-spheres.

\subsubsection*{Seifert examples}
From the previous remarks, it follows that the correction terms $d_{\pm 1/2}$ cannot show that a Seifert manifold with the homology of $S^1\times S^2$ has $DS> 1$. Using an obstruction based on the classical Rohlin invariant instead, we prove the following.

\begin{theoremalpha}\label{theorem:B}
	Let $\{N_k\}_{k\geq 1}$ be the family of $3$-manifolds described by the surgery diagram in Figure \ref{figure:OS}. Then
	\begin{enumerate}
		\item $N_k$ is irreducible with first homology $\zee$ and $\pi_1(N_k)$ of weight $1$.
		\item $N_k$ is a Seifert manifold over $S^2$ with three exceptional fibers.
		\item If $k$ is odd, $N_k$ is not obtained by Dehn surgery on a knot in $S^3$.
		\item If $k$ is odd, $N_k$ is not homology cobordant to Dehn surgery along a knot in $S^3$.
	\end{enumerate}
\end{theoremalpha}

\begin{figure}[htb!]
	\centering
	\includegraphics[scale=1]{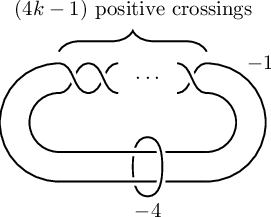}
	\caption{A surgery diagram of $N_k$ $(k\geq 1)$.}
	\label{figure:OS}
\end{figure}

Independent of questions involving weight or homology cobordism, our results provide the first known examples of irreducible homology $S^1\times S^2$'s which are not homeomorphic to surgery on  a knot in $S^3$.   To clarify the literature, it is worth mentioning here that in \cite[Section~10.2]{Ozsvath-Szabo:2003-2} Ozsv\'{a}th and Szab\'{o} argued based on the correction term obstruction that the manifold $N_1$ shown in Figure \ref{figure:OS} is not the result of Dehn surgery on a knot in $S^3$. Unfortunately, as we mentioned above, since $N_1$ is Seifert fibered the correction terms do not actually provide obstructions to $DS = 1$. We point out in Section \ref{section:Ozsvath-Szabo} where their calculation goes astray.

\begin{organization} In the next section, we establish some topological results on spliced manifolds which we will apply to our examples $M_k$.  In Section~\ref{section:background}, we briefly recall the relevant background on Heegaard Floer correction terms and the zero surgery obstruction of Ozsv\'{a}th and Szab\'{o}. Section~\ref{section:dofM_k} is devoted to computation of the correction terms of $M_k$, whose values imply they are not zero surgery on knots in $S^3$ and have the stated homology cobordism properties.  In Section \ref{section:Seifert} we prove the estimates on the correction terms of Seifert manifolds and finish the proof of Theorem \ref{theorem:A}. Section \ref{section:rokhlin} shows how the Rohlin invariant gives a different obstruction to $DS = 1$, and in Section \ref{section:Ozsvath-Szabo} we prove Theorem \ref{theorem:B}.
\end{organization}

\begin{acknowledgement}The authors would like to thank  Marco Golla and Jennifer Hom for their helpful comments and encouragement. Especially, Marco Golla gave several valuable comments on the correction terms of $N_1$ which are reflected in Section~\ref{section:Ozsvath-Szabo}. Part of this work was done while  Min Hoon Kim was visiting Michigan State University, and he thanks MSU for its generous hospitality and support.  We also thank the University of Virginia for supporting an extended visit by Matthew Hedden in 2007, which led to the discovery of the manifolds in Theorem \ref{theorem:A}.   Matthew Hedden's work on this project was partially supported by NSF CAREER grant DMS-1150872,  DMS-1709016, and an NSF postdoctoral fellowship. Min Hoon Kim was partially supported by the POSCO TJ Park Science Fellowship. Thomas Mark was supported in part by a grant from the Simons Foundation (523795, TM). Kyungbae Park was partially supported by Basic Science Research Program through the National Research Foundation of Korea (NRF, F2018R1C1B6008364).
\end{acknowledgement}

\section{Some topological preliminaries}\label{section:topology}	

In this section we verify the topological features---irreducibility and weight one fundamental group---of the manifolds $M_k$ in Theorem \ref{theorem:A}.   These features are  consequences of the fact that the manifolds are obtained by a splicing operation. Thus we establish some general results for manifolds obtained through this construction.

Given two oriented $3$-manifolds with torus boundary, $X_1,X_2$, we will refer to any manifold obtained from them by identifying their boundaries by an orientation reversing diffeomorphism as a {\em splice} of $X_1$ and $X_2$.  Of course the homeomorphism type of a splice depends intimately on the choice of diffeomorphism, but this choice will be irrelevant for the topological results that follow.  Note that with this definition Dehn filling is a splice with the unknot complement in $S^3$.  We begin with the following observation, which indicates that the manifolds appearing in Theorem \ref{theorem:A} are splices.  

\begin{proposition}\label{prop:splice}  Let $L$ be the result of connected summing the components of the Hopf link with knots $K_1$ and $K_2$, respectively.  Then any integral surgery on $L$ is a splice of the complements of $K_1$ and $K_2$.
\end{proposition}
\begin{proof} The connected sum operation can be viewed as a splicing operation.  More precisely, the connected sum of a link component with a knot $K$ is obtained by removing a neighborhood of the  meridian of the component and gluing the complement of $K$ to it by the diffeomorphism which interchanges  longitudes and meridians.  Thus the result of integral surgery on $L$ is diffeomorphic to integral surgery on the Hopf link, followed by the operation of gluing the complements of $K_1$ and $K_2$ to the complements of the meridians of the Hopf link.  But the meridians of the components of the Hopf link, viewed within the surgered manifold, are isotopic to the cores of the surgery solid tori since the surgery slopes are integral.  Thus, upon removing the meridians, we arrive back at the complement of the Hopf link, which is homeomorphic to $T^2\times [0,1]$.  The manifold at hand, then, is obtained by gluing the boundary tori of the complements of $K_1$ and $K_2$ to the boundary components of a thickened torus.  The result follows immediately.
\end{proof}

We next prove that splices of knot complements in the 3-sphere have fundamental groups of weight one.  This follows from a basic result about pushouts of groups.

\begin{proposition}\label{prop:pushoutweight} Suppose that $G_1$ and $G_2$ are groups which are normally generated by elements $g_1$ and $g_2$, respectively, and that $\phi_i:H\rightarrow G_i$ are homomorphisms.  If the image of $\phi_1$ contains $g_1$, then the pushout $G_1\ast_H G_2$ is normally generated by a single element\textup{;} namely, the image of $g_2$ under the defining map $G_2\rightarrow G_1\ast_H G_2$.
\end{proposition}
\begin{proof}  In the pushout, $g_1=\phi_1(x)=\phi_2(x)$.  Now $\phi_2(x)\in G_2$, hence can be written as a product of conjugates of $g_2$.  Since $g_1$ normally generates $G_1$, it follows that $g_2$ normally generates the pushout.
\end{proof}

It follows at once from van Kampen's theorem that  that any splice of  complements of knots in the 3-sphere has weight one fundamental group.  Indeed, the Wirtinger presentation shows that the fundamental group of a knot complement has weight one, normally generated by a meridian.  The homotopy class of the meridian is represented by a curve on the boundary, thereby verifying the hypothesis of the proposition.  Of course this reasoning shows more generally that the splice of a knot complement in $S^3$ with {\em any} manifold with torus boundary and weight one fundamental group also has fundamental group of weight one.

The discussion to this point shows that the manifolds $M_k$, being splices of knot complements, have weight one fundamental groups.  We turn our attention to their irreducibility.  As above, we will deduce this property from a more general result about splicing.

\vskip0.1in
Recall that a $3$-manifold is {\em irreducible} if any smoothly embedded $2$-sphere bounds a $3$-ball, and a surface $T$ in a 3-manifold is {\em incompressible} if any  embedded disk $D$ in the manifold for which $D\cap T=\partial D$ has the property that $\partial D$ bounds a disk in $T$ as well.

\begin{proposition}\label{prop:irreduciblesplice} Let $X_1$, $X_2$ be irreducible manifolds, each with an incompressible torus as boundary. Then any splice of $X_1$ and $X_2$ is irreducible.
\end{proposition}
The proposition applies to the complements of non-trivial knots in the $3$-sphere, which are irreducible by Alexander's characterization of the $3$-sphere \cite{Alexander:1924} (namely, that any smooth $2$-sphere separates into two pieces, each diffeomorphic to a ball), and have incompressible boundary whenever the knot is non-trivial.

\begin{proof}  The proposition follows from a standard ``innermost disk" argument.    More precisely, let $S$ be an embedded  $2$-sphere in a splice of $X_1$ and $X_2$, and let $T$ denote the image of the boundary tori, identified within the splice.  Then $S$ intersects $T$ in a collection of embedded circles.  We claim that we can remove these circles by an isotopy of $S$.  This claim would prove the proposition since,  after the isotopy, the sphere lies entirely in $X_1$ or $X_2$, where it bounds a ball by  hypothesis.  
	
	To remove the components of $S\cap T$, consider a disk $D\subset S$ which intersects  $T$  precisely in $\partial D$ (a so-called ``innermost disk", which must exist by compactness of $S\cap T$ and the Jordan-Sch{\"o}nflies theorem).   Since $D\cap T=\partial D$, the interior of $D$ must lay entirely in one of $X_1$ or $X_2$.  Incompressibility of the boundary of these manifolds therefore implies $\partial D$ bounds a disk embedded in $T$.  The union of this latter disk with $D$ is an embedded sphere  in either $X_1$ or $X_2$, which bounds a ball by its irreducibility.  The ball can be used to isotope $S$ and remove the circle of intersection.  Inducting on the number of such circles implies our claim.  
\end{proof}

\section{Heegaard Floer theory and Ozsv\'{a}th-Szab\'{o}'s 0-surgery obstruction}\label{section:background}
In this section we briefly recall the Heegaard Floer correction terms and an obstruction they yield, due to Ozsv\'ath and Szab\'o, to a 3-manifold being obtained by $0$-surgery on a knot in $S^3$. For more detailed exposition, we refer the reader to \cite{Ozsvath-Szabo:2003-2}.

Let $\mathbb{F}$ be the field with two elements, and $\mathbb{F}[U]$ be the polynomial ring over $\mathbb{F}$. Let $Y$ be a closed oriented 3-manifold endowed with a spin$^c$ structure  $\mathfrak{s}$. Heegaard Floer homology associates to the pair $(Y,\mathfrak{s})$ several relatively graded modules over $\mathbb{F}[
U]$, $HF^\circ(Y,\mathfrak{s})$, where $\circ\in\{-,+,\infty\}$. These Heegaard Floer modules are related by a long exact sequence:
\begin{equation*}
\cdots\rightarrow HF^-(Y,\mathfrak{s})\xrightarrow{\iota} HF^\infty(Y,\mathfrak{s})\xrightarrow{\pi} HF^+(Y,\mathfrak{s})\rightarrow\cdots.
\end{equation*}
The reduced Floer homology, denoted  $HF^+_{red}(Y,\mathfrak{s})$, can be defined either as the cokernel of $\pi$ or the kernel of $\iota$ with grading shifted up by one.  

In the case that the spin$^c$-structure $\mathfrak{s}$ has torsion first Chern class, the relative grading of the corresponding Floer homology modules can be lifted to an \emph{absolute} $\Q$-grading. In particular,  
$HF^\circ(Y,\mathfrak{s})$ is an absolutely $\Q$-graded $\mathbb{F}[U]$-module for any $\circ\in\{-,+,\infty\}$. 

For a rational homology 3-sphere $Y$, every spin$^c$ structure will have torsion Chern class, and we define the \emph{correction term} $d(Y,\mathfrak{s})\in \Q$ to be the minimal $\Q$-grading of any element in $HF^+(Y,\mathfrak{s})$ in the image of $\pi$. A structure theorem \cite[Theorem 10.1]{Ozsvath-Szabo:2004-2} for the Floer modules  states that $HF^\infty(Y,\mathfrak{s})\cong\mathbb{F}[U,U^{-1}]$, from which it follows that  \[HF^+(Y,\mathfrak{s})\cong\mathcal{T}^+_{d(Y,\mathfrak{s})}\oplus HF^+_{red}(Y,\mathfrak{s}),\] 
where $\mathcal{T}^+_d$ denotes the $\Q$-graded $\mathbb{F}[U]$-module isomorphic to $\mathbb{F}[U,U^{-1}]/U\mathbb{F}[U]$ in which the non-trivial element with lowest grading occurs in grading $d\in \Q$. Multiplication by $U$ decreases the $\Q$-grading by $2$.

A 3-manifold $Y$ with $H_1(Y;\Z)\cong\Z$ has  a unique spin$^c$ structure  with torsion (zero) Chern class;  we denote this spin$^c$ structure by $\mathfrak{s}_0$. In this setting, the structure theorem states that $HF^\infty(Y,\mathfrak{s})\cong\mathbb{F}[U,U^{-1}]\oplus\mathbb{F}[U,U^{-1}] $, with the two summands supported in grading $\pm \frac{1}{2}$ modulo 2, respectively.  We  define $d_{1/2}(Y)$ and  $d_{-1/2}(Y)$ to be the minimal grading of any  element in the image of $\pi$ in $HF^+(Y,\mathfrak{s}_0)$ supported in the grading $\frac{1}{2}$ and $-\frac{1}{2}$ modulo $2$, respectively. It follows that 
\begin{equation*}
HF^+(Y,\mathfrak{s}_0)\cong\mathcal{T}^+_{d_{-1/2}(Y)}\oplus\mathcal{T}^+_{d_{1/2}(Y)}\oplus HF^+_{red}(Y,\mathfrak{s}_0).
\end{equation*}

The key features of the correction terms are  certain constraints they place on  negative semi-definite 4-manifolds bounded by a given 3-manifold, \cite[Theorem~9.11]{Ozsvath-Szabo:2003-2}.  Applying these constraints to the 4-manifold obtained from a homology cobordism by drilling out a neighborhood of an arc connecting the boundaries yields the following (compare \cite[Proposition~4.5]{Levine-Ruberman:2014-1}):

\begin{Homology Cobordism Invariance} If $Y$ and $Y'$ are integral homology cobordant homology manifolds with first homology $\Z$, then $d_{\pm 1/2}(Y)=d_{\pm 1/2}(Y')$.
\end{Homology Cobordism Invariance}
The relevance to the surgery question at hand also becomes apparent: if a 3-manifold $Y$ is obtained by 0-surgery on a knot $K$ in $S^3$,  then $Y$ bounds a homology $S^2\times D^2$, gotten by attaching a $0$-framed 2-handle to the 4-ball along $K$, and so does $-Y$ after reversing the orientation of the 4-manifold.  Coupling this observation with the constraints mentioned above, and using the fact that $d_{-1/2}(Y)=-d_{1/2}(-Y)$ \cite[Proposition~4.10]{Ozsvath-Szabo:2003-2}, we get the following obstruction:

\begin{zerosurgeryobstruction}[{\cite[Corollary 9.13]{Ozsvath-Szabo:2003-2}}] If $Y$ bounds a homology $S^2\times D^2$  then $d_{1/2}(Y)\leq\frac{1}{2}$ and $d_{-1/2}(Y)\geq -\frac{1}{2}$.  
\end{zerosurgeryobstruction}
\noindent The obstruction applies, for instance, if $Y$ is homology cobordant to zero surgery on a knot in a 3-manifold that bounds a smooth contractible 4-manifold.

Drawing on information from the surgery exact triangle, Ozsv\'{a}th and Szab\'{o} \cite[Proposition~4.12]{Ozsvath-Szabo:2003-2} gave a refined statement of the obstruction, which determines the values of the correction terms.  We rephrase their result in terms of  the non-negative knot invariant  $V_0(K)$  introduced by Rasmussen  (under the name $h_0(K)$) in \cite{Rasmussen:2003-1}, and used by Ni-Wu \cite{Ni-Wu:2015-1}.  To see that the following agrees with the stated reference, we recall that $d(S^3_1(K))=-2V_0(K)$, and that the $d$-invariant of $1$-surgery changes sign under orientation reversal (implying $d(S^3_{-1}(K))=2V_0(\overline{K})$).

\begin{theorem}[{\cite[Proposition 4.12]{Ozsvath-Szabo:2003-2}}]\label{thm:OS}Suppose that $Y$ is obtained by $0$-surgery on a knot $K$ in~$S^3$. Then $d_{1/2}(Y)=\frac{1}{2}-2V_0(K)$ and $d_{-1/2}(Y)=-\frac{1}{2}+2V_0(\overline{K})$ where $\overline{K}$ is the mirror of $K$.
\end{theorem}

\section{Computation of $d_{\pm 1/2}(M_k)$}\label{section:dofM_k}
Consider the 3-manifold $M_k$ obtained by $(1,1)$ surgery on the link obtained from the Hopf link by connected summing one component with the right-handed trefoil $T_{2,3}$ and the other component with the $(2,4k-1)$ torus knot $T_{2,4k-1}$ as depicted in Figure~\ref{figure:Hedden-Mark}. In this section, we compute $d_{\pm 1/2}(M_k)$ for any $k\geq 1$.  We assume that the reader is familiar with knot Floer homology \cite{Ozsvath-Szabo:2004-1,Rasmussen:2003-1}.

\begin{theorem}\label{thm:dofM_k}For any $k\geq 1$, $d_{1/2}(M_k)=-2k+\frac{1}{2}$ and $d_{-1/2}(M_k)=-\frac{5}{2}$.
\end{theorem}

We briefly discuss the strategy of our computation. Consider the knot $J_k$ in $S^3_{1}(T_{2,3})$ depicted in Figure~\ref{figure:J_k}. Since $H_1(M_k)\cong \Z$, $M_k$ is the result of surgery on $S^3_1(T_{2,3})$ along the knot $J_k$ using its Seifert framing.  Note that the Seifert framing of $J_k$ is the 1-framing with respect to the blackboard framing of Figure~\ref{figure:J_k}. Then $d_{\pm 1/2}(M_k)$ can be determined by the knot Floer homology $CFK^\infty(S^3_1(T_{2,3}),J_k)$ using a surgery formula \cite[Section~4.8]{Ozsvath-Szabo:2008-1}.

\begin{figure}[tb!]
	\centering
	\includegraphics[scale=1]{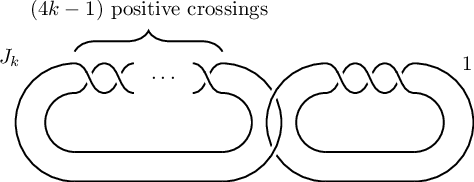}
	\caption{A knot $J_{k}$ in $S^3_1(T_{2,3})$.}
	\label{figure:J_k}
\end{figure}

In order to determine the aforementioned knot Floer homology complex, we first consider the  meridian of $T_{2,3}$, viewed as a knot $\mu\subset S^3_1(T_{2,3})$. Then   the relevant knot $(S^3_1(T_{2,3}),J_k)$ is simply the connected sum of two knots, $(S^3_1(T_{2,3}),\mu)$ and $(S^3,T_{2,4k-1})$. A K\"{u}nneth formula for the knot Floer homology of connected sums then implies 
\begin{equation}\label{equation:Kunneth}CFK^\infty(S^3_1(T_{2,3}),J_k)\cong CFK^\infty (S^3_1(T_{2,3}),\mu)\otimes CFK^\infty (T_{2,4k-1}).\end{equation}

We can deduce the structure of  $CFK^\infty(S^3_1(T_{2,3}),\mu)$ using a surgery formula which,  together with the K\"{u}nneth formula and the well-known structure of the Floer homology of torus knots, will determine the filtered chain homotopy type of $CFK^\infty(S^3_1(T_{2,3}),J_k)$.   Precisely, we prove the following:

\begin{proposition}\label{proposition:CFKofJ_k}We have the following filtered chain homotopy equivalences.
	\begin{enumerate}
		\item\label{item:CFKofmeridian} $CFK^\infty(S^3_1(T_{2,3}),\mu)\cong CFK^\infty (T_{2,-3})[-2]$.
		\item\label{item:CFKofJ_k} $CFK^\infty(S^3_{1}(T_{2,3}),J_k)\oplus A_0\cong CFK^\infty(T_{2,4k-3})[-2]\oplus A_1$.
	\end{enumerate}
	Here $[-2]$ means that the Maslov grading is shifted by $-2$, and $A_0$ and $A_1$ are  acyclic chain complexes over $\mathbb{F}[U,U^{-1}]$.
\end{proposition}

\begin{remark}For $N\geq 2g(K)$, it is known that $CFK^\infty(S^3_{-N}(K),\mu)$ is determined by $CFK^\infty(S^3,K)$ in \cite[Theorem~4.2]{Hedden-Kim-Livingston:2016-1} (compare  \cite[Theorem~4.1]{Hedden:2007-1}). Since $1<2g(T_{2,3})$, we cannot apply \cite[Theorem~4.2]{Hedden-Kim-Livingston:2016-1} to determine $CFK^\infty(S^3_1(T_{2,3}),\mu)$.  Work in progress of Hedden and Levine on a general surgery formula for the knot Floer homology of $\mu$ would easily yield the formula.  In the case at hand, however, a surgery formula applied for $\mu$ allows for an ad hoc argument.\end{remark}

\begin{proof} (\ref{item:CFKofmeridian}) The key observation is that the complement of $\mu \subset S^3_1(T_{2,3})$ is homeomorphic to the complement of  $T_{2,3}\subset S^3$.  Indeed,  this can be seen by observing that $\mu$ is isotopic to the core of the surgery solid torus.  It follows that $\mu$ is a genus one fibered knot.  Moreover, $S^3_1(T_{2,3})$ is homeomorphic to the Poincar{\'e} sphere equipped with the opposite orientation it inherits as the boundary of the resolution  of the surface singularity $z^2+w^3+r^5=0$, which is well known and easily seen to be an $L$-space homology sphere with $d$-invariant equal to $-2$.  
	
	As $\mu$ is a genus one fibered knot in an $L$-space homology sphere, it follows readily that its knot Floer homology must have rank $5$ or $3$, and in the latter case must be isomorphic to that of one of the trefoil knots, with an overall shift in the Maslov grading by the $d$-invariant.  To see this, observe that being genus one implies, by the adjunction inequality for knot Floer homology \cite[Theorem~5.1]{Ozsvath-Szabo:2004-1}, that  $\widehat{HFK}(S^3_1(T_{2,3}),\mu,i)=0$ for $|i|>1$.  As $\mu$ is fibered, $\widehat{HFK}(S^3_1(T_{2,3}),\mu,i)=\mathbb{F}$ for $i=\pm1$  \cite[Theorem~5.1]{Ozsvath-Szabo:2005-1}.  Moreover, the Maslov grading of the generator of $\widehat{HFK}(S^3_1(T_{2,3}),\mu,1)$ is two higher than that of $\widehat{HFK}(S^3_1(T_{2,3}),\mu,-1)$, by a symmetry of the knot Floer homology groups \cite[Proposition~3.10]{Ozsvath-Szabo:2004-1}.   Now there is a differential $\d$ acting on $\widehat{HFK}(-S^3_1(T_{2,3}),\mu)$, the homology of which is isomorphic to the Floer homology of the ambient 3-manifold. (The existence of such a ``cancelling differential" follows from the homological method of reduction of a filtered chain complex; see \cite[Section 2.1]{Hedden-Watson:2018-1} for details on this perspective.) This differential strictly lowers the Alexander grading, which implies that the ``middle" group $\widehat{HFK}(S^3_1(T_{2,3}),\mu,0)$ is either $\mathbb{F}^3$ or $\mathbb{F}$. We discuss the cases separately (compare \cite[Proposition 3.1]{Baldwin:2008-1}).

	If $\widehat{HFK}(S^3_1(T_{2,3}),\mu,0)=\mathbb{F}^3$, two of the summands are supported in the same grading, which is one less than that of the top group; moreover, one of these summands is  the image of $\widehat{HFK}(S^3_1(T_{2,3}),\mu,1)$ under~$\d$, and $\d$ maps the other summand  surjectively onto $\widehat{HFK}(S^3_1(T_{2,3}),\mu,-1)$.   This follows immediately from existence of the cancelling differential.  The remaining summand of $\mathbb{F}^3$ lies in Maslov grading $-2$, the $d$-invariant of the underlying manifold.  If this happens to be the grading of the other two summands, then the resulting knot Floer homology is thin, and $CFK^\infty$ is determined by the hat groups.  It follows in this case that $CFK^\infty$ is isomorphic to that of the figure-eight knot, with an overall shift in the Maslov grading down by $2$.

	The case that $\widehat{HFK}(S^3_1(T_{2,3}),\mu,0)=\mathbb{F}$ divides into two sub-cases, depending on whether $\d$  maps the middle group surjectively onto the bottom, or the top group surjectively onto the middle. In both sub-cases the resulting knot Floer homology is thin, and hence $CFK^\infty$ is determined by the hat groups.  In the former sub-case the hat groups are isomorphic to those of the right-handed trefoil,  and to those of the left-handed trefoil in the latter; in both sub-cases, their Maslov grading has an overall shift down by $2$.  
	
	To determine which of the three possibilities above arise, we recall the surgery formula for knot Floer homology.  In its simplest guise, which will be sufficient for our purposes, it expresses the Floer homology of the manifold obtained by $n$-surgery, $n\le -(2g(K)-1)$ on a null-homologous knot $(Y,K)$ as the homology of a particular sub-quotient complex of $CFK^\infty(Y,K)$  \cite[Theorem~4.1]{Ozsvath-Szabo:2004-1}.   Its relevance to us is that $-1$-surgery on $(S^3_1(T_{2,3}),\mu)$ is homeomorphic to $S^3$, a manifold with $\widehat{HF}(S^3)$ of rank $1$.   Since $\mu$ is a genus one knot, we can apply the surgery formula to (re)-calculate the Floer homology of $S^3$, viewed as $-1$-surgery on $\mu$.  The surgery formula says that the homology is given as the homology of the subquotient complex of $CFK^\infty(S^3_1(T_{2,3}),\mu)$ generated by chains whose $\Z\oplus\Z$-filtration values satisfy the constraint min$(i,j)=0$.   Of the three possibilities for $\widehat{HFK}(S^3_1(T_{2,3}),\mu)$, all but the case of the left-handed trefoil (shifted down in grading by $2$) have the property that the relevant subquotient complex has homology of rank $3$. Indeed, $-1$-surgery on the right-handed trefoil or figure-eight knots have Floer homology of rank $3$.  In the case that the middle group of knot Floer homology has rank $3$ but is not supported in a single grading, the fact that $\partial^2=0$ on $CFK^\infty$ implies the two arrows in the subquotient complex must have the same head.  This, in turn,  forces the homology of the subquotient to have rank 3.  The stated structure of $CFK^\infty(S^3_1(T_{2,3}),\mu)$ now follows.
	
	(\ref{item:CFKofJ_k}) 
	By (\ref{item:CFKofmeridian}), the K{\"u}nneth formula (\ref{equation:Kunneth}) becomes 
	\begin{equation}\label{equation:J_k}CFK^\infty(S^3_1(T_{2,3}),J_k)\cong CFK^\infty (T_{2,-3})[-2]\otimes CFK^\infty (T_{2,4k-1}).\end{equation}
	For brevity, we say two chain complexes $C_0$ and $C_1$ are \emph{stably filtered chain homotopy equivalent} (denoted by $C_0\sim C_1$) if $C_0\oplus A_0$ is filtered chain homotopy equivalent to $C_1\oplus A_1$ for some acyclic chain complexes $A_0$ and $A_1$.  By \cite[Theorem B.1]{Hedden-Kim-Livingston:2016-1}, 
	\[CFK^\infty(T_{2,4k-1})\sim CFK^\infty(T_{2,3})\otimes CFK^\infty(T_{2,4k-3}).\]
	By \cite[Proposition 3.11]{Hom:2017-1},
	\[CFK^\infty(T_{2,-3})\otimes CFK^\infty(T_{2,3})\cong CFK^\infty (T_{2,-3}\# T_{2,3})\sim CFK^\infty(U)\]
	since $T_{2,-3}\# T_{2,3}$ is slice. Hence $CFK^\infty(T_{2,-3})\otimes CFK^\infty(T_{2,4k-1})\sim CFK^\infty(T_{2,4k-3})$.
	It follows that the right hand side of \eqref{equation:J_k} is stably filtered chain homotopy equivalent to $CFK^\infty(T_{2,4k-3})[-2]$, and we obtain the desired conclusion.
\end{proof}

Now we prove Theorem \ref{thm:dofM_k} which states that $d_{1/2}(M_k)=-2k+\frac{1}{2}$ and $d_{-1/2}(M_k)=-\frac{5}{2}$ if $k\geq 1$.
\begin{proof}[Proof of Theorem \ref{thm:dofM_k}] Recall that $M_k$ is obtained from $S^3_1(T_{2,3})$ by surgery on $J_k$ along its Seifert framing. In Section 6.2 of \cite{Hedden-Kim-Livingston:2016-1} it is shown that the $d$-invariants of large surgery on knots with stably filtered homotopy equivalent complexes agree. Indeed,  \cite[Proposition 6.5]{Hedden-Kim-Livingston:2016-1} shows that direct summing an acyclic complex to a given one has no effect on the $d$-invariants  one  derives from it.  The $d$-invariants of large surgery on a knot are equivalent to the $V_i$ invariants, hence we obtain
	\begin{align*}
	&d_{1/2}(M_k)=d_{1/2}(S^3_0(T_{2,4k-3}))-2=-\tfrac{3}{2}-2V_0(T_{2,4k-3}),\\
	&d_{-1/2}(M_k)=d_{-1/2}(S^3_0(T_{2,4k-3}))-2=-\tfrac{5}{2}+2V_0(T_{2,-4k+3}).
	\end{align*}
	by Proposition \ref{proposition:CFKofJ_k}(\ref{item:CFKofJ_k}) and Theorem \ref{thm:OS}. Strictly speaking, Theorem~\ref{thm:OS} pertains only to surgery on knots in $S^3$, but the proof easily yields a corresponding formula for surgery on knots in an integral homology sphere $L$-space; in these cases, the correction terms inherit an overall shift by the $d$-invariant of the ambient manifold (here, $-2$).   Since $k\geq 1$,  
	\begin{align*}&V_0(T_{2,4k-3})=k-1,\\
	&V_0(T_{2,-4k+3})=0\end{align*} (for example, see \cite[Theorem~1.6]{Borodzik-Nemethi:2013-1}). This completes the proof.
\end{proof}

\section{Correction terms of Seifert manifolds}\label{section:Seifert}

In this section we provide some general constraints on the  correction terms of  a Seifert fibered homology $S^1\times S^2$.    More precisely, we show  $d_{-1/2}(M)\geq -\frac{1}{2}$ and $d_{1/2}(M)\leq\frac{1}{2}$ for any  Seifert fibered homology $S^1\times S^2$.  It follows at once that none of our manifolds are homology cobordant to a Seifert fibered space.  We also see that the zero surgery obstruction can say nothing about Seifert manifolds.

\vskip0.1in

Our estimates hinge on the following proposition, which was pointed out to us by Marco Golla. (Compare \cite[Theorem~5.2]{Neumann-Raymond:1978-1}.)

\begin{proposition}\label{proposition:Seifertboundsboth}
	Suppose $M$ is a Seifert fibered homology $S^1\times S^2$. Then both $M$ and $-M$ bound negative semi-definite, plumbed $4$-manifolds.  
\end{proposition}
\begin{proof}
	Choose an orientation and a Seifert fibered structure of $M$. As an oriented manifold, $M$ is homeomorphic to $M(e;r_1,\ldots,r_n)$ in Figure~\ref{figure:Seifertmanifold} where $e$ is an integer, and each $r_i$ is a non-zero rational number. We change the Seifert invariant $(e;r_1,\ldots,r_n)$ via the following two steps.
	
	\begin{figure}[b]
		\centering
		\includegraphics[scale=1]{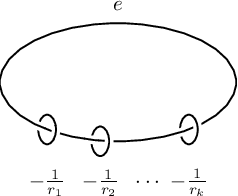}
		\caption{A Seifert fibered 3-manifold $M(e;r_1,\ldots,r_k)$.}
		\label{figure:Seifertmanifold}
	\end{figure}
	
	\begin{enumerate}
		\item If $r_i$ is an integer, remove $r_i$ from the tuple $(e;r_1,\ldots,r_n)$ and add $r_i$ to $e$.
		\item For each $i$, replace $r_i$ and $e$ by $r_i-\lfloor r_i\rfloor$ and $e+\lfloor r_i\rfloor$, respectively.
	\end{enumerate}
	Note that the above procedures are realized by slam-dunk moves, so the homeomorphism type remains unchanged. For brevity, we still denote the resulting Seifert invariant of $M$ by $(e;r_1,\ldots,r_n)$, so that each rational number $r_i$ satisfies $0<r_i<1$. Since $0<r_i<1$, we can write $-\frac{1}{r_i}$ as a negative continued fraction $[a_{i1},\ldots,a_{ik_i}]$ where $a_{ij}\leq -2$ for all $i$ and $j$. Then $M$ bounds a star-shaped plumbed 4-manifold $X_{\Gamma}$ whose corresponding plumbing graph is $\Gamma$ depicted in Figure~\ref{figure:plumbingofSeifert}.
	
	\begin{figure}[h]
		\centering
		\includegraphics[scale=0.75]{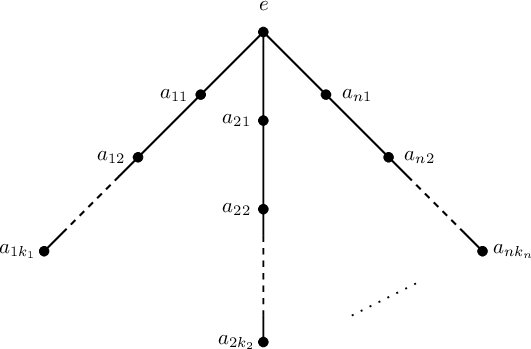}
		\caption{A plumbing graph $\Gamma$.}
		\label{figure:plumbingofSeifert}
	\end{figure}
	
	Since $a_{ij}\leq -2$ for all $i$ and $j$, it is easy to check that $Q_{X_\Gamma}$ is negative semi-definite. (Since $\d X_\Gamma=M$ is a homology $S^1\times S^2$ and $\Gamma$ is a tree, $Q_{X_\Gamma}$ has determinant $0$.)
\end{proof}
\begin{remark}
	The plumbed 4-manifold $X_\Gamma$ constructed in the proof of Proposition~\ref{proposition:Seifertboundsboth} is called the \emph{normal form} of $M$. (Note that $X_\Gamma$ depends only on the choice of orientation of $M$.) What we have shown in Proposition~\ref{proposition:Seifertboundsboth} is that the normal forms of $M$ and $-M$ are negative semi-definite plumbings. (Compare \cite[Theorem~5.2]{Neumann-Raymond:1978-1} where it is shown that one  normal form of a Seifert fibered rational homology sphere is a negative-definite plumbing.)
\end{remark}
We recall a special case of \cite[Corollary~4.8]{Levine-Ruberman:2014-1}. (Note that if $M$ is a closed, oriented 3-manifold with $H_1(M)\cong \Z$, then $M$ has standard $HF^\infty$, and $d_{-1/2}(M)$ is equal to $d(M,\mathfrak{s}_0,H_1(M))$ with the notation of \cite[Corollary~4.8]{Levine-Ruberman:2014-1}.) We remark that this special case essentially follows from \cite[Theorem~9.11]{Ozsvath-Szabo:2003-2} and Elkies' theorem \cite{Elkies:1995-1}.

\begin{proposition}[{\cite[Corollary~4.8]{Levine-Ruberman:2014-1}}]\label{proposition:Levine-Ruberman} Let $M$ be a closed, oriented $3$-manifold with first homology $\Z$. Suppose that $M$ bounds a negative semi-definite, simply connected $4$-manifold $X$. Then $d_{-1/2}(M)\geq -\frac{1}{2}$.
\end{proposition}

\begin{theorem}\label{theorem:dofSeifert}Suppose $M$ is homology cobordant to a Seifert fibered homology $S^1\times S^2$. Then $d_{-1/2}(M)\geq -\frac{1}{2}$ and $d_{1/2}(M)\leq \frac{1}{2}$.
\end{theorem}
\begin{proof}Since $d_{-1/2}$ and $d_{1/2}$ are homology cobordism invariants, we can assume that $M$ is a Seifert fibered homology $S^1\times S^2$. By Proposition~\ref{proposition:Seifertboundsboth}, both $M$ and $-M$ bound negative semi-definite, plumbed 4-manifolds. Since plumbed 4-manifolds are simply connected, we can apply Proposition~\ref{proposition:Levine-Ruberman} to conclude that $d_{-1/2}(M)\geq -\frac{1}{2}$, and $d_{-1/2}(-M)\geq -\frac{1}{2}$. Since $d_{-1/2}(-M)=-d_{1/2}(M)$, the desired conclusion follows.
\end{proof}

\begin{proof}[Proof of Theorem \ref{theorem:A}] From the surgery diagram of $M_k$ in Figure~\ref{figure:Hedden-Mark}, it is easy to compute that $H_1(M_k)\cong \Z$. By Proposition~\ref{prop:splice}, $M_k$ is a splice of non-trivial knot complements which, by Propositions~\ref{prop:irreduciblesplice} and \ref{prop:pushoutweight}, implies that $M_k$ is irreducible and has weight 1 fundamental group.  Ozsv{\'a}th and Szab{\'o}'s obstruction, Theorem~\ref{thm:OS}, combined with  our calculation of the correction terms, Theorem~\ref{thm:dofM_k},  shows that $M_k$ is not the result of Dehn surgery on a knot in $S^3$. Since $d_{1/2}$ is a homology cobordism invariant, $M_k$ and $M_l$ are not homology cobordant if $k\neq l$ by Theorem~\ref{thm:dofM_k}. Finally,  Theorems~\ref{thm:dofM_k} and~\ref{theorem:dofSeifert} show that $M_k$ is not homology cobordant to any Seifert fibered 3-manifold. This completes the proof.
\end{proof}

\section{Rohlin invariant and another surgery obstruction}\label{section:rokhlin}

While the Heegaard Floer correction terms provide an obstruction for a 3-manifold to arise as 0-surgery on a knot in $S^3$, we see by comparing Theorem~\ref{theorem:dofSeifert} with the zero-surgery obstruction of Section~\ref{section:background} that they cannot show a Seifert manifold is not 0-surgery on a knot (one could view the zero-surgery obstruction and the Seifert constraint Theorem~\ref{theorem:dofSeifert} as arising from the same observation, that in both cases the 3-manifold in question bounds negative semi-definite with both orientations). In this section we observe that the classical Rohlin invariant can obstruct a homology $S^1\times S^2$ from having surgery number~1, and we will see that this obstruction can be effective in the Seifert case. 

Recall that if $(Y,s)$ is a spin 3-manifold, the Rohlin invariant $\mu(Y,s)\in {\mathbb Q}/2\zee$ is defined to equal $\frac{1}{8}\sigma(X)$ modulo $2$, where $X$ is a compact 4-manifold with $\partial X = Y$ that admits a spin structure extending $s$. If $Y$ is a homology sphere then $Y$ admits a unique spin structure, and since a spin 4-manifold with boundary $Y$ has signature divisible by 8, we have $\mu(Y)\in \zee/2\zee$. If $Y_0$ has the homology of $S^1\times S^2$ then $Y_0$ has two spin structures, and hence two Rohlin invariants (each also with values in $\zee/2\zee$). 

\begin{lemma}\label{lemma:rokhlincalc} Let $Y$ be an integral homology sphere and $K\subset Y$ a knot\textup{;} write $Y_0(K)$ for the result of $0$-framed surgery along $K$. Then the Rohlin invariants of $Y_0(K)$ are equal to $\mu(Y_0(K), \mathfrak{s}_0) = \mu(Y)$ and $\mu(Y_0(K), \mathfrak{s}_1) = \mu(Y) + \arf(K)$.
\end{lemma}

\begin{proof} Let $X$ be a spin 4-manifold with boundary $Y$. The obvious 0-framed 2-handle cobordism $W$ from $Y$ to $Y_0(K)$ carries a (unique) spin structure, and if $\mathfrak{s}_0$ is the spin structure on $Y_0(K)$ induced by the one on the cobordism, then $X\cup_Y W$ is a spin 4-manifold with spin boundary $(Y_0(K),\mathfrak{s}_0)$ and the same signature as $X$. Hence $\mu(Y_0(K), \mathfrak{s}_0) = \mu(Y)$. 
	
	It is not hard to see that the other spin structure on $Y_0(K)$ is spin cobordant (by a 0-framed surgery cobordism) to the unique spin structure on $Y_1(K)$, the result of $1$-framed surgery on $K$ (see, for example, Section 5.7 of \cite{gompfstipsicz}). The same argument as above implies $\mu(Y_0(K), \mathfrak{s}_1) = \mu(Y_1(K))$. On the other hand, the surgery formula for the Rohlin invariant (as in \cite[Theorem 2.10]{saveliev}) implies $\mu(Y_1(K)) = \mu(Y) + \arf(K)$.
\end{proof}

One could also phrase the lemma as the statement that the Rohlin invariants of $Y_0(K)$ are equal to $\mu(Y)$ and $\mu(Y_1(K))$. Since $\mu(S^3) = 0$, we infer:

\begin{corollary} If $Y_0$ is a $3$-manifold obtained by $0$-framed surgery on a knot in $S^3$, then at least one of the Rohlin invariants of $Y_0$ vanishes. The other Rohlin invariant is equal to the Arf invariant of any knot $K\subset S^3$ such that $Y_0 = S^3_0(K)$.
\end{corollary}

\begin{corollary}\label{corollary:0surgeryobstruction} If an integral homology $S^1\times S^2$ has two nontrivial Rohlin invariants, then it is not obtained by surgery on a knot in $S^3$. 
\end{corollary}

\section{Properties of the manifolds $N_k$}\label{section:Ozsvath-Szabo}
In this section, we discuss the manifolds $N_k$ given by the surgery diagrams of Figure~\ref{figure:OS}.
\begin{proposition}\label{proposition:Seifert}For any positive integer $k$, $N_k$ is an irreducible Seifert fibered $3$-manifold.
\end{proposition}
\begin{proof} We can see a Seifert fibered structure of $N_k$ from the sequence of Kirby moves depicted in Figure~\ref{figure:Ozsvath_Szabo_Kirbydiagram}. We remark that a similar sequence of Kirby moves is given in Lemma 2.1 of~\cite{Lisca-Stipsicz:2007-2}. By Figure~\ref{figure:Ozsvath_Szabo_Kirbydiagram}, $N_k$ admits a Seifert fibering $ N_k\to S^2$. Note that the slopes $r_i$ of the exceptional fibers are $\frac{8k-3}{16k-2}$, $\frac{1}{8k-1}$ and $\frac{1}{2}$, respectively. It is known that any orientable, reducible Seifert fibered 3-manifold is homeomorphic to either $S^1\times S^2$ or $\R\mathbb{P}^3\#\R\mathbb{P}^3$ (for example, see \cite[Lemma~VI.7]{Jaco:1980}). Since $H_1(N_k)\cong \Z$, $N_k$ is not homeomorphic to $\R\mathbb{P}^3\#\R\mathbb{P}^3$. By the homeomorphism classification of Seifert fibered 3-manifolds \cite{Seifert}, we can conclude that $N_k$ is not homeomorphic to $S^1\times S^2$, and hence $N_k$ is irreducible.
\end{proof}

\begin{figure}[h]
	\centering
	\includegraphics[width=0.9\textwidth]{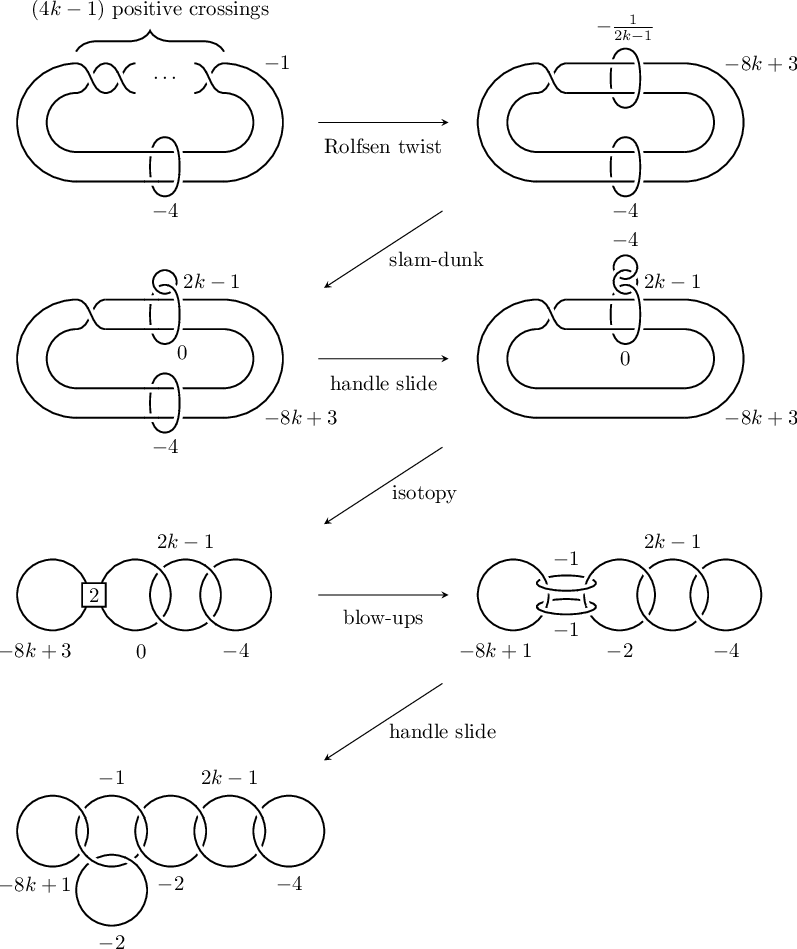}
	\caption{A Seifert fibered structure of $N_k$.}
	\label{figure:Ozsvath_Szabo_Kirbydiagram}
\end{figure}

\begin{proposition}The weight of $\pi_1(N_k)$ is one for any positive integer $k$.
\end{proposition}
\begin{proof}We observed above that $N_k$ is a Seifert fibered 3-manifold with 3 exceptional fibers whose slopes are  $\frac{8k-3}{16k-2}$, $\frac{1}{8k-1}$ and $\frac{1}{2}$.  Therefore we have a presentation of $\pi_1(N_k)$ as follows (compare \cite[page 91]{Jaco:1980}):
	\begin{align*}\pi_1(N_k)&\cong \langle x_1,x_2,x_3,h\mid x_1^{16k-2}=h^{8k-3}_{\vphantom{1}}, x_2^{8k-1}=h, x_3^2=h, x_1x_2x_3=h, [h,x_i]=1\rangle\\
	&\cong\langle x_1,x_2,h\mid x_1^{16k-2}=h^{8k-3}, x_2^{8k-1}=h, h=x_1x_2x_1x_2, [h,x_1]=[h,x_2]=1\rangle\\
	&\cong \langle x_1^{\vphantom{16k-2}},x_2^{\vphantom{16k-2}}\mid x_1^{16k-2}=x_2^{(8k-1)(8k-3)}, x_2^{8k-1}=(x_1^{\vphantom{8k-1}}x_2^{\vphantom{8k-1}})^2, [x_2^{8k-1},x_1^{\vphantom{8k-1}}]=1\rangle.
	\end{align*}
	In the second equality, we cancel the generator $x_3$ with the relation $x_3^{\vphantom{-1}}=x_2^{-1}x_1^{-1}h$. Note that the relation $x_3^2=h$ is equivalent to the relation $x_2^{-1}x_1^{-1}hx_2^{-1}x_1^{-1}h=h$, and hence to the relation $h=x_1x_2x_1x_2$. In the last equality, we cancel the generator $h$ and the relation $h=x_2^{8k-1}$.
	
	Let $\langle \langle  x_2^{4k-2}x_1^{-1}\rangle\rangle$ be the normal subgroup of $\pi_1(N_k)$ generated by $x_2^{4k-2}x_1^{-1}$. Then
	\begin{align*} \pi_1(N_k)/\langle \langle x_2^{4k-2}x_1^{-1}\rangle\rangle
	&\cong \langle x_1^{\vphantom{16k-2}},x_2^{\vphantom{16k-2}}\mid x_1^{16k-2}=x_2^{(8k-1)(8k-3)}, x_2^{8k-1}=(x_1^{\vphantom{8k-1}}	x_2^{\vphantom{8k-1}})^2, x_1^{\vphantom{4k-2}}=x_2^{4k-2}\rangle\\
	&\cong \langle x_2\mid x_2^{(8k-1)(8k-4)}=x_2^{(8k-1)(8k-3)}, x_2^{8k-1}=x_2^{8k-2}\rangle\\
	&\cong \langle x_2\mid x_2^{(8k-1)(8k-4)}=x_2^{(8k-1)(8k-3)}, x_2=1\rangle=1. 
	\end{align*}
	Hence, the weight of $\pi_1(N_k)$ is 1.
\end{proof}

Since $N_k$ is Seifert, the correction terms do not provide information on the surgery number of $N_k$; instead we  apply the obstruction of Corollary \ref{corollary:0surgeryobstruction} of the previous section. To do so we must calculate the Rohlin invariants of the two spin structures on $N_k$. One way to make this calculation, along the lines of the previous section, is to observe that $N_k$ can be realized as the result of nullhomologous surgery on the singular Seifert fiber of order $4k-1$ in the Brieskorn homology sphere $\Sigma(2,4k-1,8k-1)$, which has Rohlin invariant equal to $k$ modulo 2. Performing surgery on that fiber with framing $+1$ gives another plumbed 3-manifold whose Rohlin invariant is also $k$ modulo 2, and these two calculations give the desired invariants for $N_k$ by the remark after Lemma \ref{lemma:rokhlincalc}. 

Alternatively, one can proceed directly from the final diagram in Figure \ref{figure:Ozsvath_Szabo_Kirbydiagram} using the algorithm in \cite[Section~6]{Neumann-Raymond:1978-1} (see also \cite[Section~4]{Neumann:1979} for the case with nonzero first homology), as follows. If $P_k$ denotes the plumbed 4-manifold described by the last diagram of Figure \ref{figure:Ozsvath_Szabo_Kirbydiagram}, we can find exactly two homology classes $\nu_1,\nu_2\in H_2(P_k;\zee/2)$, represented by embedded spheres or a disjoint union thereof, satisfying $\nu_i.x = x.x\pmod 2$ for each homology class $x$. (Here the dot indicates the intersection product.) These ``spherical Wu classes'' give the Rohlin invariants of the two spin structures on $N_k$ by the formula $\mu(N_k, \mathfrak{s}_i) = \frac{1}{8}(\sigma(P_k) - \nu_i.\nu_i)\pmod 2$, where $\sigma(P_k)$ is the signature of the intersection form on $P_k$.

The two Wu classes on $P_k$ are given by letting $\nu_1$ be the sum of the spheres represented by the circles with framings $-8k+1$ and $-4$, and taking $\nu_2$ as the sum of the $-8k+1$ sphere with the two $-2$ spheres. It is straightforward to check that $P_k$ has $b^+(P_k) = 1$ and hence $\sigma(P_k) = -3$, while $\nu_1.\nu_1 = \nu_2.\nu_2 = -8k-3$. Hence the Rohlin invariants $\mu(N_k, \mathfrak{s}_i)$ are both equal to $k \pmod 2$, and we conclude: 

\begin{theorem}\label{theorem:Nk} For any odd integer $k\geq 1$, the manifold $N_k$ is a Seifert fibered integral homology $S^1\times S^2$ that cannot be obtained by surgery on a knot in $S^3$.
\end{theorem}

Moreover, since the Rohlin invariant is unchanged under integral homology cobordism, we have that when $k$ is odd, no $N_k$ is homology cobordant to a 3-manifold with $DS(Y) =1$. This concludes the proof of Theorem~\ref{theorem:B}.

\begin{figure}[h]
	\centering
	\includegraphics[scale=0.9]{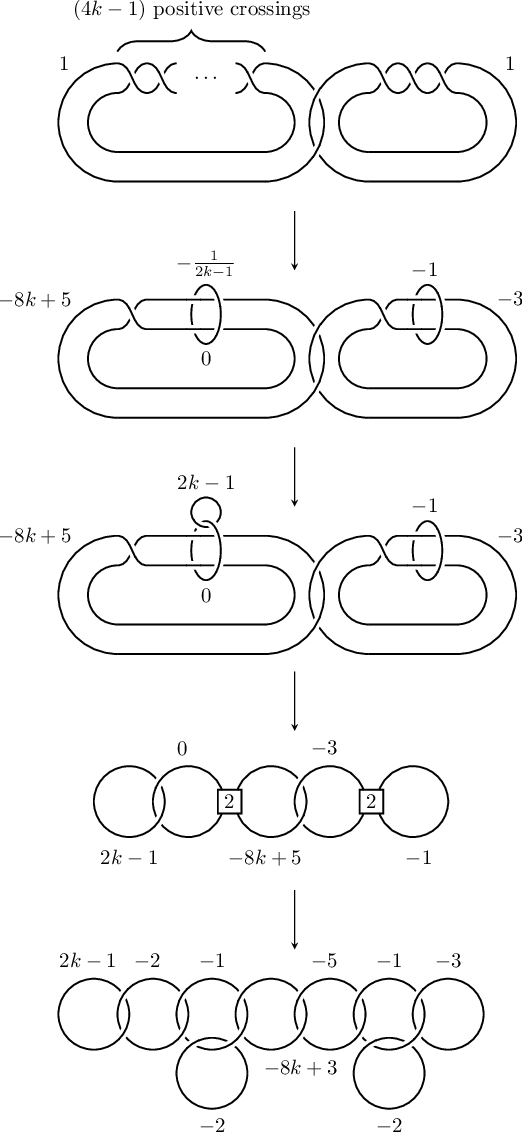}
	\caption{A plumbing diagram of $M_k$.}
	\label{figure:Hedden_Mark_plumbingdiagram}
\end{figure}
\begin{remark}It is interesting to see if the obstruction from Corollary~\ref{corollary:0surgeryobstruction} can be applied to the manifolds $M_k$ ($k\geq 1$) in Figure~\ref{figure:Hedden-Mark}. One can compute Rohlin invariants of $M_k$ by using a plumbing description of the manifold given in Figure~\ref{figure:Hedden_Mark_plumbingdiagram}. (Since the computation is similar to the above, we leave this to the reader.) When $k$ is odd, the Rohlin invariants of $M_k$ are both non-trivial, and hence this also obstruct $M_k$ from being Dehn surgery on a knot in $S^3$. On the other hand, when $k$ is even, one Rohlin invariant of $M_k$ is trivial, so the obstruction from Corollary~\ref{corollary:0surgeryobstruction} is ineffective in this case.
\end{remark}

As noted in the introduction, the manifold $N_1$ is the example from Ozsv\'{a}th and Szab\'{o} \cite[Section 10.2]{Ozsvath-Szabo:2003-2}, where the case $k=1$ of Theorem \ref{theorem:Nk} was claimed based on an argument using the Heegaard Floer correction terms. As we have seen, the correction terms do not provide an obstruction to $DS = 1$ in the case of Seifert manifolds. Here we revisit the calculation of $d_{\pm 1/2}(N_1)$ from \cite{Ozsvath-Szabo:2003-2}, which relies on the surgery exact triangle and some understanding of the maps therein.  In particular, they fit the Floer homology of $N_1$ in an exact triangle between two lens spaces, $L(49,40)$ and $L(49,44)$, and identify a spin structure on the 2-handle cobordism between $L(49,40)$ and $N_1$.  The map on Floer homology associated to this spin structure, which is summed along with those associated to the other spin$^c$ structures, induces an isomorphism between submodules of $HF^\infty$ isomorphic to $\mathbb{F}[U,U^{-1}]$.  It follows that that the ``tower" in $HF^+$ of the spin structure on $L(49,40)$ surjects onto the tower of $HF^+(N_1)$ relevant to $d_{-1/2}$.  From this, and the grading shift by $-\frac{1}{2}$ for the map induced by the spin structure, one concludes an inequality \[d_{-1/2}(N_1)\ge d(L(49,40),\mathfrak{s}_0)-\tfrac{1}{2}=-\tfrac{5}{2}.\]
Here $d(L(49,40),\mathfrak{s}_0)$ is the correction term for the spin structure, which is easily seen to be $-2$.  This inequality is opposite to the one inferred by Ozsv\'{a}th and Szab\'{o}.  A rather detailed examination of the exact triangle shows that the bottommost element of the tower associated to the spin structure lies in the kernel of the sum of  the cobordism maps involved in the exact triangle (which is, of course, the only way for $d_{-1/2}(N_1)> -\frac{5}{2}$).

We conclude with  an alternate proof that $d_{-1/2}(N_k)\geq -\frac{1}{2}$ and $d_{1/2}(N_k)\leq \frac{1}{2}$. First recall a result of Ozsv\'{a}th and Szab\'{o}.
\begin{proposition}[{\cite[Corollary~9.14]{Ozsvath-Szabo:2003-2}}]\label{theorem:Ozsvath-Szaboinequality} Suppose that $K$ is a knot in a homology $3$-sphere~$Y$. Let $Y_0$ be the result of Dehn surgery along $K$ via its Seifert framing. Then 
	\[d_{1/2}(Y_0)-\tfrac{1}{2}\leq d(Y)\leq d_{-1/2}(Y_0)+\tfrac{1}{2}.\]
\end{proposition}

\begin{figure}[h]
	\centering
	\includegraphics[scale=1]{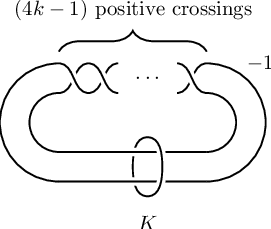}
	\caption{A knot $K$ in $S^3_{-1}(T_{2,4k-1})$.}
	\label{figure:K_k}
\end{figure}
\begin{proposition}\label{proposition:dofN_k}For any positive integer $k$, $d_{-1/2}(N_k)\geq -\frac{1}{2}$ and $d_{1/2}(N_k)\leq \frac{1}{2}$.
\end{proposition}

\begin{proof}Consider the knot $K\subset S^3_{-1}(T_{2,4k-1})$ which is depicted in Figure~\ref{figure:K_k}. By a surgery formula given in \cite{Ni-Wu:2015-1}, $d(S^3_{-1}(T_{2,4k-1}))=2V_0(T_{2,-4k+1})=0$ since $k\geq 1$. Then $N_k$ is the result of Dehn surgery along the knot $K\subset S^3_{-1}(T_{2,4k-1})$ via its Seifert framing. By Theorem~\ref{theorem:Ozsvath-Szaboinequality}, we have 
	\[d_{1/2}(N_k)-\tfrac{1}{2}\leq 0\leq d_{-1/2}(N_k)+\tfrac{1}{2}\]
	for any positive integer $k$. This completes the proof.
\end{proof}

\bibliographystyle{amsalpha}
\bibliography{research}

\end{document}